\documentclass[12pt]{amsart}
\usepackage{amsmath,amssymb,amscd,amsthm,amsxtra}
\usepackage{latexsym}
\usepackage{graphicx}

\usepackage[colorlinks=true, pdfstartview=FitV, linkcolor=blue, citecolor=blue, urlcolor=blue]{hyperref}
\usepackage{color}

\usepackage{graphicx,epstopdf,verbatim}

\allowdisplaybreaks

\newtheorem{theorem}{Theorem}[section]
\newtheorem{lemma}[theorem]{Lemma}
\newtheorem{corollary}[theorem]{Corollary}

\newcommand{\q}{\quad}
\newcommand{\qq}{\quad\quad}

\newcommand{\al}{\alpha}
\newcommand{\be}{\beta}

\newcommand{\de}{\delta}

\newcommand{\ez}{\epsilon}

\newcommand{\la}{\lambda}

\newcommand{\vp}{\varphi}

\newcommand{\note}[1]{\vskip.3cm
\fbox{%
\parbox{0.93\linewidth}{\footnotesize #1}} \vskip.3cm}

\def\rr{{\mathbb R}}

\numberwithin{equation}{section}

\begin{document}

\title[ Sharp bounds for general commutators]{Sharp bounds  for general commutators
on weighted Lebesgue spaces}

\author{Daewon Chung, Mar\'{\i}a Cristina Pereyra and Carlos Perez.}

\begin{abstract}
We show that if a linear operator $T$ is bounded  on weighted Lebesgue
space $L^2(w)$ and obeys a linear bound with respect to the $A_2$ constant
of the weight, then its commutator  $[b,T]$ with a function $b$ in $BMO$ will obey
a quadratic bound with respect to the $A_2$ constant of the weight.
We  also prove that the $k$th-order commutator
$T^k_b=[b,T^{k-1}_b]$ will obey a bound that is a power $(k+1)$
of the $A_2$ constant of the weight. Sharp extrapolation provides corresponding $L^p(w)$ estimates.
In particular these estimates hold for $T$ any Calder\'on-Zygmund singular
integral operator. The results are sharp in terms of the growth of the operator norm with
respect to the $A_p$ constant of the weight for all $1<p<\infty$,
all $k$, and all dimensions, as examples involving the Riesz transforms,
 power functions and power weights show.
\end{abstract}

\address{Daewon Chung\\
Department of Mathematics and Statistics
MSC01 1115\\ 1 University of New Mexico\\
Albuquerque, NM 87131-0001} \email{midiking@math.unm.edu}

\address{Mar\'{\i}a Cristina Pereyra\\
Department of Mathematics and Statistics\\
MSC01 1115\\ 1 University of New Mexico\\
Albuquerque, NM 87131-0001}
 \email{crisp@math.unm.edu}

\address{Carlos P\'erez\\
Departamento de An\'alisis Matem\'atico, Facultad de Matem\'aticas,
Universidad De Sevilla, 41080 Sevilla, Spain.}
\email{carlosperez@us.es}



\thanks{The third author would like
to acknowledge the support of Spanish Ministry of Science and
Innovation via grant MTM2009-08934.} \maketitle



\section{Introduction}

Singular integral operators are known to be bounded in weighted
Lebesgue spaces $L^p(w)$ if the weight belongs to the $A_p$ class of
Muckenhoupt.

Recently there has been renewed interest in understanding the
dependence of the operator norm in  terms of the $A_p$ constant of
the weight, more precisely one seeks estimates of the type,
$$
\|T\|_{L^p(w)} \leq \vp_p([w]_{A_p})\qquad 1<p<\infty,
$$
where the  function $\vp_p:[1,\infty)\to [0,\infty)$  is optimal
in terms of growth.  The first result of this type was obtained by
S. Buckley \cite{B} who showed that the maximal function obeyed
such estimates with $\vp_p(t)=c_pt^{\frac{1}{p-1}}$ for
$1<p<\infty$, and this is optimal (see \cite{Le1} for another
recent proof). This problem has attracted renewed attention
because of the work of Astala, Iwaniec and Saksman \cite{AIS}.
They proved sharp regularity results for solutions to the Beltrami
equation, assuming that the operator norm of the Beurling-Ahlfors
transform grows linearly in terms of the $A_p$ constant for $p\geq
2$. This linear growth
 was proved by S. Petermichl and A.
Volberg \cite{PetV} and by Petermichl \cite{Pet1,Pet2} for the
Hilbert transform and the Riesz Transforms. In these papers it has
been shown that if $T$ is any of these operators, then
\begin{equation}\label{A2conjec}
\|T\|_{L^p(w)}\le
c_{p,n}\,[w]_{A_p}^{\max\left\{1,\frac{1}{p-1}\right\}}
\qquad1<p<\infty,
\end{equation}
and the exponent $\max\left\{1,\frac{1}{p-1}\right\}$ is best
possible. It has been conjectured, and very recently proved \cite{Hyt},
 that the same estimate holds for
any Calder\'on-Zygmund operator $T$. By the sharp version of the
Rubio de Francia extrapolation theorem \cite{DGPerPet}, it suffices to
prove this inequality for $p=2$, namely
\begin{equation}\label{CZA2conjecture}
\|T\|_{L^2(w)}\le c_{n}\, [w]_{A_2}.
\end{equation}
%
We remit the reader to \cite{CrMP2} for a new proof and for generalizations and applications. The linear growth in $L^2(w)$ has been shown to hold for
dyadic operators (martingale transform \cite{Wi}, dyadic square
function \cite{HTV}, dyadic paraproduct \cite{Be}), or for operators
who have lots of symmetries and can be written as averages of dyadic
shift operators (such us the Hilbert transform \cite{Pet1}, Riesz
transforms \cite{Pet2},  Beurling transform \cite{PetV}, \cite{DV}).
All  these estimates were obtained using Bellman functions. Recently all
the above results have been recovered using different sets of
techniques,  and obtaining linear bounds for  \textcolor{blue}{a} larger class of
Haar shift operators \cite{LPetR},  \textcolor{red}{
\cite{CrMP0}. } In particular, in the latter paper  \cite{CrMP0}, no  Bellman function techniques nor
any two weight results are used, and the methods can be extended to other important
operators in Harmonic Analysis such as dyadic square
functions and paraproducts, maximal singular integrals and the
vector-valued maximal function as can be found in \cite{CrMP1}. The sharp bound \eqref{CZA2conjecture}  for any Calder\'on-Zygmund operator $T$ has been proved in \cite{Hyt} by T. Hyt\"onen.  Hyt\"onen's proof is based on  approximating
$T$ by generalized dyadic Haar shift operators with good bounds combined with the key fact that to prove \eqref{CZA2conjecture} it is enough to prove the corresponding weak type $(2,2)$ estimate with the same linear bound as proved in \cite{PTV}. A direct proof avoiding this weak $(2,2)$ reduction can be found in \cite{HytPTV}.  A bit earlier, in \cite{Le2}, the sharp $L^p(w)$ bound for $T$ was obtained for values of $p$ outside the interval $(3/2,3)$ and the proof is based on the corresponding estimates for the intrinsic square functions.

It should be mentioned that until the $A_2$-conjecture was proved,
 only the following special case
\begin{equation}\label{A1resultforT}
\|T\|_{L^p(w)} \leq C_p\,[w]_{A_1} \qquad 1<p<\infty,
\end{equation}
had been shown to be true for any Calder\'on-Zygmund operator.
Observe that the condition imposed on the weight is the $A_1$ weight
condition which is stronger than $A_p$ but there is a gain in the
exponent since it is linear for any $1<p<\infty$ (compare with
\eqref{A2conjec}). This has been shown in \cite{LOP1,LOP2} and we
remit the reader to \cite{P5}  for a survey on this topic.

The main purpose of this paper is to prove estimates similar to
\eqref{A2conjec} for commutators of appropriate linear operators $T$
with $BMO$ functions $b$. These operators are defined formally by
the expression
$$[b,T]f=b T(f) - T(b\,f). $$
When $T$ is a singular integral operator, these operators were
considered by Coifman, Rochberg and Weiss in \cite{CRW}. Although
the original interest in the study of such operators was related to
generalizations of the classical factorization theorem for Hardy
spaces many other applications have been found.


The main result from \cite{CRW} states that $[b, T]$ is a bounded
operator on $L^{p}( \mathbb R^{n} )$, $1<p<\infty$, when $b$ is a
$BMO$ function and $T$ is a singular integral operator. In fact, the
$BMO$ condition of $b$ is also a necessary condition for the
$L^{p}$-boundedness of the commutator when $T$ is the Hilbert
transform. Later on a different proof was given by J. O. Str\"omberg
(cf. \cite{Tor} p.417) with the advantage that it allows to show
that these commutators are also bounded on weighted $L^p(w)$, when $w\in A_p$.
This approach is based on the use of the classical C.
Fefferman-Stein maximal function and it is not precise enough for
further developments. Indeed, we may think that these operators
behave as Calder\'on-Zygmund operators, however there are some
differences. For instance, an interesting fact is that, unlike what
it is done with singular integral operators, the proof of the
$L^{p}$-boundedness of the commutator does not rely on a weak type
$(1,1)$ inequality.  In fact, simple examples show that in general
$[b,T]$ fails to be of weak type $(1,1)$ when $b \in BMO$. This was
observed by the third author in \cite{P1} where it is also shown
that there is an appropriate weak-$L(\log L)$ type estimate
replacement. This shows that the operator cannot be a
Calder\'on-Zygmund singular integral operator. To stress this point
of view it is also shown by the third author \cite{P2} that the
right operator controlling $[b,T]$ is $M^2=M\circ M$, instead of the
Hardy-Littlewood maximal function $M$.


In the present paper we pursue this point of view by showing that
commutators have an extra ``bad" behavior from the point of view of
the $A_p$ theory of weights that it is not reflected in the
classical situation. Our argument will be based on
the second proof for the $L^p$-boundedness of the commutator
presented in \cite{CRW}. This proof is interesting because there is
no need to assume that $T$ is a singular integral operator, to show
the boundedness of the commutator it is enough to assume that the
operator $T$ is linear and bounded on $L^p(w)$ for any $w\in A_p$.
These ideas were exploited in \cite{ABKP}.

The first author showed in \cite{Ch} that the commutator with the
Hilbert transform obeys an estimate of the following type,
\begin{equation}\label{A2resultforComm}
\| [b,H] \|_{L^2(w)} \leq C\, [w]_{A_2}^2\|b\|_{BMO},
\end{equation}
and he also showed that the quadratic growth with respect to the
$A_2$ constant of the weight is sharp. The techniques used in that
paper rely  very much in dyadic considerations  and the
use of Bellman function arguments. Using recent
results on Haar shifts operators, he also deduced the quadratic
growth for commutators of Haar shift operators and operators in their
convex hull, including the Riesz transforms and the Beurling-Ahlfors
operator, see \cite{Ch} for the

Also, it should be mentioned that there is a corresponding version of  \eqref{A1resultforT} for commutators of any Calder\'on-Zygmund operator with quadratic growth as in \eqref{A2resultforComm}. This is proved in \cite{O} where an endpoint estimate can also be found.


By completely different methods, we show in this paper that if an
operator obeys an initial linear bound in $L^2(w)$, then its
commutator will obey a quadratic bound in $L^2(w)$. In light
of the positive resolution of the $A_2$-conjecture,  we conclude that
the commutator of any Calder\'on-Zygmund singular integral operator
and a $BMO$ function obeys a quadratic bound in $L^2(w)$.  In fact we show
that if an operator $T$ obeys a bound in $L^2(w)$ of the form $\vp
([w]_{A_2})$, then its $k$-th order commutator with $b\in BMO$,
$T_b^k:= [b, T^{k-1}_b]$, will obey a bound of the form $c_n^k k!
\vp(\gamma_n[w]_{A_2}) [w]^k_{A_2} \|b\|^k_{BMO}$. Observe that if
we consider the special case of Calder\'on-Zygmund operators  with
kernel $K$ then
$$ T_b^k(f)(x)\int_{\mathbb
R^n} (b(x)-b(y))^kK(x,y)f(y)\,dy, $$
and the larger $k$ is,  the more singular the operator will be,
because the exponent in $[w]^k_{A_2}$ becomes larger.


Corresponding estimates in $L^p(w )$ are deduced by the sharp
version of the Rubio de Francia extrapolation theorem found in
\cite{DGPerPet}, and are shown to be sharp in the case of the Hilbert
and Riesz transforms (in any dimension) for all $1<p<\infty$, and
for all $k\geq 1$.

It will be interesting to recover the result for the higher order
commutators with the Haar shift operators (and hence for the Hilbert,
Riesz and Beurling transforms) using the dyadic methods, but so far
we do not know how to do this.

Recently extensions of our result to two weight settings, fractional integrals and more
have been obtained by D. Cruz-Uribe and Kabe Moen, see \cite{CrM}.

The remainder of this paper is organized as follows.  In Section
\ref{section:prelim} we gather some basic results. In Section
\ref{proofmainresult} we give the proof of the main result. In
Section \ref{Examples} we show with examples that the main theorem
in the paper,
and its corollaries are sharp. Finally, the last section is an appendix where we
show a result that it is claimed but never proved in the literature,
a sharp reverse H\"older's inequality for $A_2$ weights.


\section{Preliminary results} \label{section:prelim}

\subsection{A Sharp John-Nirenberg}

For a locally integrable $b:\rr^n\to \rr$ we define
$$ \|b\|_{BMO} = \sup_Q \frac{1}{|Q|}\int_{Q} |b(y)-
b_{Q}|\,dy <\infty ,$$
where the supremum is taken over all cubes $Q\in\rr^n$ with sides
parallel to the axes, and
$$b_{Q}= \frac1{|Q|}\int_Q b(y)\,dy.$$

The main relevance of $BMO $ is because of its exponential
self-improving property, recorded in the celebrated
John-Nirenberg Theorem \cite{JN}. We need a very precise version of
it, as follows:

\begin{theorem}\label{sharpJ-N}[Sharp John-Nirenberg]
There are dimensional constants $0\leq \al_n<1<\be_n$ such that
\begin{equation}\label{sharpJN}
\sup_Q \frac{1}{  |Q| } \int_{Q} \exp \left (
\frac{\al_n}{\|b\|_{BMO}} |b(y)- b_{Q}| \right  )\, dy \le \be_n.
\end{equation}
In fact we can take $\al_n=\frac{1}{2^{n+2}}$.
\end{theorem}
For the proof of this we remit to p. 31-32 of \cite{J} where a proof
different from the standard one can be found.

We derive from Theorem~\ref{sharpJ-N} the following Lemma~\ref{BMO&A2}.
 that will be used in the proof
of the main theorem. First  recall that a weight $w$ \, satisfies
the \,$A_2$\, condition if
$$
[w]_{A_2}=  \sup_Q \left (\frac{1}{|Q|}\int_{Q}w \right )\left (\frac{1}{|Q|}\int_{Q}w
^{-1}\right ) <\infty,
$$
where the supremum is taken over all cubes $Q\in\rr^n$ with sides
parallel to the axes. Notice that $[w]_{A_2}\geq 1$.
%

It is well known that if $w\in A_2$  then $b=\log w \in BMO$. A
partial converse also holds, if $b\in BMO$ there is an $s_0>0$ such
that $w=e^{sb}\in A_p, \, |s|\leq s_0$. As a consequence of the
Sharp John-Nirenberg Theorem  we can get a more precise version of
this partial converse.

\begin{lemma}\label{BMO&A2} Let $b\in BMO$ and let $\al_n<1<\be_n$  be the dimensional constants
from~\eqref{sharpJN}. Then
$$
s\in  \rr, \q |s| \leq \frac{\al_n}{\|b\|_{BMO}} \Longrightarrow
e^{s\,b}\in A_2 \q \mbox{and} \q [e^{s\,b}]_{A_2}\leq \be_n^2.
$$
\end{lemma}

\begin{proof} 

By Theorem \ref{sharpJ-N}, if $|s| \leq \frac{\al_n}{\|b\|_{BMO}}$
and if $Q$ is fixed
$$
\frac{1}{  |Q| } \int_{Q} \exp ( |s||b(y)- b_{Q}|  )\, dy \le
\frac{1}{  |Q| } \int_{Q} \exp ( \frac{\al_n}{\|b\|_{BMO}}|b(y)-
b_{Q}|  )\, dy \leq \be_n,
$$
thus
$$
\frac{1}{  |Q| } \int_{Q} \exp ( s(b(y)- b_{Q})  )\, dy \leq \be_n,
$$
and
$$
\frac{1}{  |Q| } \int_{Q} \exp ( -s(b(y)- b_{Q})  )\, dy \leq \be_n.
$$
If we multiply the inequalities, the $b_Q$ parts cancel out:
$$
\left (\frac{1}{  |Q| } \int_{Q} \exp ( s(b(y)- b_{Q})  )\, dy
\right )\,\left (\frac{1}{ |Q| } \int_{Q} \exp ( s(b_{Q}-b(y))  )\,
dy \right )
$$
$$
= \left (\frac{1}{  |Q| } \int_{Q} \exp ( sb(y)  )\, dy \right ) \,
\left( \frac{1}{  |Q| } \int_{Q} \exp ( -sb(y)  )\, dy\right )\leq
\be_n^2
$$
namely $e^{s\,b} \in A_2$ with
$$
[e^{s\,b}]_{A_2}\leq \be_n^2.
$$

\end{proof}

 We remark that it follows easily from minor modifications to the proof of
 Lemma~\ref{BMO&A2} that if $1<p<\infty$
$$
s\in  \rr, \q |s| \leq \frac{\al_n}{\|b\|_{BMO}  }\min
\left\{1,\frac{1}{p-1}\right\} \Longrightarrow e^{s\,b}\in A_p \q
\mbox{and} \q [e^{s\,b}]_{A_p}\leq \be_n^{p},
$$
where as usual
$$[w]_{A_p}= \sup_Q\left(\frac{1}{|Q|}\int_Qw(x)dx\right)\left(\frac{1}{|Q|}\int_Q
w(x)^{-1/(p-1)}dx\right)^{p-1}<\infty.$$

\subsection{Sharp reverse H\"older inequality for the $A_2$ class of weights}

Recall that if $w\in A_2$  then $w$ satisfies a reverse H\"older
condition, namely, there are constants $r>1$ and $c\ge1$ such that
for any cube $Q$
\begin{equation}\label{a2}
\left ( \frac{1}{|Q|}\int_Qw^{r}dx\right )^{\frac1r} \le
\frac{c}{|Q|}\int_Qw
\end{equation}
In the usual proofs, both constants, $c$ and $r$, depend upon the
$A_2$ constant of the weight. There is a more precise version of
\eqref{a2}.

\begin{lemma}\label{RHA2} Let $w\in A_2$ and let
$r_{w}=1+\frac{1}{2^{n+5}[w]_{A_2}}$. Then
$$\left(
\frac{1}{|Q|}\int_Qw^{r_w}dx\right )^{\frac{1}{r_w}} \le
\frac{2}{|Q|}\int_Qw $$
\end{lemma}

This result was stated and used in \cite{B} but no proof was given.
The author mentioned instead the celebrated work \cite{CF} where no
explicit statement can be found. We supply  in Section
\ref{append} a proof taken from \cite{P5}, where a more general
version can be found as well as more information.


\section{Main result}\label{proofmainresult}

\begin{theorem}\label{main}
Let $T$ be a linear operator bounded on $L^2(w)$ for any $w\in A_2$.
Suppose further that there is an increasing function  $\vp:[1,\infty)\to
[0,\infty)$ such that
\begin{equation}
\|T\|_{L^{2}(w)} \le \vp([w]_{A_2}).
\end{equation}
then there are constants $\gamma _n$ and $c_n$ independent of
\,$[w]_{A_2}$ such that
\begin{equation}\label{sharpcommA2}
\|[b,T]\|_{L^{2}(w)} \le c_n\, \vp(\gamma_n\,[w]_{A_2})\,[w]_{A_2}
\|b\|_{BMO}.
\end{equation}
\end{theorem}

For the particular case $\vp (t)=c_0t^r$, where $r>0$, and $c_0> 0$,
a simple induction argument shows that if
\[ \|T\|_{L^2(w)}\leq a_0\, [w]_{A_2}^r,\]
then for each integer $k\geq 1$ there is a constant $a_k$ depending
on $k$, the initial value $a_0$, and the parameters $\gamma_n$ and
$c_n$ in the theorem, such that the $k$th-order commutator $T^k_b$
defined recursively by $T_b^{k} := [b, T_b^{k-1}]$, obeys the
following weighted estimates
\[\|T_b^k\|_{L^2(w)} \leq a_k \,[w]_{A_2}^{r+k} \,\|b\|_{BMO}^k.\]
More precisely,  the sequence $\{a_k\}_{k\geq 0}$ obeys the
following recurrence equation that can be solved easily,
\[  a_k=c_n\,a_{k-1}\,\gamma_n^{r+k-1}= c_n^k\,a_0\,\gamma_n^{kr+\frac{(k-2)(k-1)}{2}}.\]
Using the method of proof of  Theorem~\ref{main} we can obtain a
weighted estimate for the $k$th-order commutator that works for general
increasing function $\vp:[1,\infty)\to [0,\infty)$. Notice  the
difference in the constants with  what we just argued by induction
for the particular case $\vp(t)=a_0t^r$: in the corollary  the
constant is $c_n^k\,a_0\,\gamma_n^r\,k!$,  whereas in the induction
argument the constant is
$c_n^k\,a_0\,\gamma_n^{kr+\frac{(k-2)(k-1)}{2}}$.

\begin{corollary}\label{COR1}
Let $T$ be a bounded linear operator on $L^2(w)$ with $w\in A_2$ and
\begin{equation}
\|T\|_{L^{2}(w)} \le \vp([w]_{A_2}).
\end{equation}
then there are constants $\gamma_n$ and $c_n$ independent of
\,$[w]_{A_2}$ such that
\begin{equation}\label{sharpgener.commA2}
\|T^k_b\|_{L^{2}(w)} \le c_n^k\, k! \, \vp(\gamma_n
[w]_{A_2})\,[w]_{A_2}^{k} \|b\|_{BMO}^k.
\end{equation}
\end{corollary}
The constants $\gamma_n$ and $c_n$ that appear in Corollary~\ref{COR1} are the same that
appeared in Theorem~\ref{main}.
We first present the proof of Theorem~\ref{main}, and afterwards
 we discuss the necessary modifications to obtain Corollary~\ref{COR1}.
As an easy consequence of Corollary \ref{COR1} and the Rubio de
Francia extrapolation theorem with sharp constants \cite{DGPerPet}, we
have the following.
\begin{corollary}\label{COR2}
Let $T$ be a linear operator bounded on $L^2(w)$ with $w\in A_2$ and
\begin{equation}
\|T\|_{L^{2}(w)} \le \vp([w]_{A_2}).
\end{equation}
Then, for $1<p<\infty$, there are constants $\gamma_{n,p}$,  and
$c_{n,p}$,  which only depend on $p$, and the dimension $n$, such that
for all weights $w\in A_p\,$
\begin{equation}\label{sharpgener.commAp}
\|T^k_b\|_{L^{p}(w)} \le \sqrt{2}\,c_{n,p}^k\,k!\, \vp\left
(\gamma_{n,p}\,[w]_{A_p}^{\max\{1,\frac{1}{p-1}\}}\right
)\,[w]_{A_p}^{k\max\{1,\frac{1}{p-1}\}} \|b\|_{BMO}^k.
\end{equation}
\end{corollary}

In the particular case $\vp (t)=a_0t^r$ the extrapolated estimate
looks like
\begin{equation}\label{extrapolationpowervarphi}
\|T^k_b\|_{L^{p}(w)} \le \sqrt{2}\, a_0\,c_{n}^k\, k!\,
\gamma_{n}^r\, c_p^{r+k}[w]_{A_p}^{(r+k)\max\{1,\frac{1}{p-1}\}}
\|b\|_{BMO}^k.
\end{equation}
where $c_p$ depends only on $p$, $\gamma_n$ and $c_n$ are the
constants that appeared in Theorem~\ref{main}.

We will show in Section~\ref{Examples} that for $r=1$, $\vp(t)=a_0t$
the  power $(1+k)\max\{1,\frac{1}{p-1}\}$ cannot be decreased for
$T=H$ and $T=R_j$ the Hilbert and Riesz transforms, for all $k\geq
1$ and $p>1$, therefore the theorem is optimal in terms of the rate
of the dependence on $[w]_{A_p}$. In \cite{Ch}, examples for $k=1$,
for all $p>1$, and for $T$ the Hilbert, Beurling and Riesz
transforms, were presented.

\begin{proof}[Proof of Theorem \ref{main}]


We ``conjugate'' the operator as follows: if $z$ is any complex
number we define
$$T_z(f)=e^{zb} T(e^{-zb}f).$$
Then, a computation gives (for instance for "nice" functions),
$$
[b,T](f)=\frac{d}{dz}T_z(f)|_{z=0}=\frac{1}{2\pi i}\int_{|z|=\ez}
\frac{T_z(f)}{z^2}\,dz\, , \qq \ez>0$$
by the Cauchy integral theorem,  see \cite{CRW}, \cite{ABKP}.

Now, by Minkowski's inequality
\begin{equation}\label{minkows}
\|[b,T](f)\|_{L^2(w)}\leq \frac{1}{2\pi\,\ez^2} \,\int_{|z|=\ez}
\|T_z(f)\|_{L^2(w)}|dz| \qq \ez>0.
\end{equation}

The key point is to find the appropriate radius $\ez$. To do this we
look at the inner norm \, $\|T_z(f)\|_{L^2(w)}$\,
$$
\|T_z(f)\|_{L^2(w)}=\|T(e^{-zb}f)\|_{L^2 (we^{2Re z\,b })} \, ,$$
and try to find appropriate bounds on $z$. To do this we use the
main hypothesis, namely that $T$ is bounded on $L^2(w)$ if $w\in
A_2$ with
$$
\|T\|_{L^{2}(w)} \le \vp([w]_{A_2}). $$
Hence we should compute
$$
[we^{2Re z\,b }]_{A_2}= \sup_Q \left (\frac1{|Q|}\int_Q we^{2Re
z\,b(x) }\,dx \right )\left (\frac1{|Q|}\int_Q w^{-1}e^{-2Re z\,b(x)
}\,dx\right ).
$$
Now, since $w\in A_2$ we use Lemma \ref{RHA2}: if \,
$r=r_{w}=1+\frac{1}{2^{n+5}[w]_{A_2}}<2\;\;$ then
$$
\left ( \frac{1}{|Q|}\int_Qw^{r}dx\right )^{\frac{1}{r}} \le
\frac{2}{|Q|}\int_Qw\, , $$
and similarly for $w^{-1}$ since $r_{w}=r_{w^{-1}}$,
$$
\left ( \frac{1}{|Q|}\int_Qw^{-r}dx\right )^{\frac{1}{r}} \le
\frac{2}{|Q|}\int_Qw^{-1}\, .
$$
Using this and Holder's inequality we have for an arbitrary $Q$
$$
\left (\frac1{|Q|}\int_Q w(x)e^{2Re z\,b(x) }\,dx \right )\left
(\frac1{|Q|}\int_Q w(x)^{-1}e^{-2Re z\,b(x) }\,dx \right )\leq
$$
$$
\left ( \frac{1}{|Q|}\int_Qw^{r}dx\right )^{\frac{1}{r}} \left (
\frac1{|Q|}\int_Q e^{2Re z\,r'\,b(x) }\,dx\right )^{\frac{1}{r'}}
\left ( \frac{1}{|Q|}\int_Qw^{-r}dx\right )^{\frac{1}{r}} \left (
\frac1{|Q|}\int_Q e^{-2Re z\,r'\,b(x) }\,dx\right )^{\frac{1}{r'}}
$$
$$
\leq 4\,\left (\frac{1}{|Q|}\int_Qw\,dx \right )\left
(\frac{1}{|Q|}\,\int_Qw^{-1}\,dx \right )\left (\frac1{|Q|}\int_Q
e^{2Re z\,r'\,b(x) }\,dx\right )^{\frac{1}{r'}}  \left (
\frac1{|Q|}\int_Q e^{-2Re z\,r'\,b(x) }\,dx\right )^{\frac{1}{r'}}
$$
$$
\leq 4\,[w]_{A_2}\,[e^{2Re z\,r'\,b}]_{A_2}^{\frac{1}{r'}}
$$

Now, since $b\in BMO$ we are in a position to apply Lemma
\ref{BMO&A2},
$$
\mbox{if} \q |2Re z\,r'| \leq \frac{\al_n}{\|b\|_{BMO}}
 \qq \mbox{then} \qq [e^{2Re z\,r'\,b}]_{A_2}\leq \be_n^2.
$$
Hence for these $z$ , and since $1<r<2$,
$$
[we^{2Re z\,b}]_{A_2}\leq 4\,[w]_{A_2}\, \be_n^{\frac{2}{r'}}\leq
4\,[w]_{A_2}\, \be_n .
$$

Using this estimate  for these $z$, and observing that
$\|e^{-zb}f\|_{L^2(we^{2Rezb})}=\|f\|_{L^2(w)}$,
$$
\|T_z(f)\|_{L^2(w)}=\|T(e^{-zb}f)\|_{L^2 (we^{2Re z\,b })}\leq
\vp([we^{2Re z\,b}]_{A_2}) \|f\|_{L^2(w)} \leq \vp(4[w]_{A_2}\,
\be_n)\,
 \|f\|_{L^2(w)} .
$$

Choosing now the radius
$$\ez= \frac{\al_n}{2r'\|b\|_{BMO}}  ,$$
we can continue estimating the norm in (\ref{minkows})
$$\|[b,T](f)\|_{L^2(w)}\leq \frac{1}{2\pi\,\ez^2}
\,\int_{|z|=\ez} \|T_z(f)\|_{L^2(w)}|dz|
$$
$$
\leq \frac{1}{2\pi\,\ez^2} \,\int_{|z|=\ez} \vp(4[w]_{A_2}\,
\be_n)\,
 \|f\|_{L^2(w)} |dz|
= \frac{1}{\ez}\,\vp(4[w]_{A_2}\, \be_n)\,\|f\|_{L^2(w)} ,
$$
since \,
$$|2Re z\,r'|\leq 2|z|\,r' = 2\ez\,r'=
 \frac{\al_n}{\|b\|_{BMO}}.$$
Finally, for this \,$\ez$,
$$\|[b,T](f)\|_{L^2(w)}\leq  C2^{2n}\,\vp(4[w]_{A_2}\, \be_n)\,[w]_{A_2}\,\|b\|_{BMO},
$$
because $r'=1+2^{n+5}[w]_{2}\approx 2^n[w]_{2}$, and
$\alpha_n=\frac{1}{2^{n+2}}$.

Observe that the optimal radius is essentially the inverse of
$[w]_{2}\|b\|_{BMO}$.  This proves the theorem with $c_n\sim
2^{2n}$ and $\gamma_n=4\,\be_n$.

\end{proof}

\begin{proof}[Proof of Corollary~\ref{COR1}]

In this case a computation gives (for instance for "nice"
functions), see \cite{ABKP} for example or the original paper
\cite{CRW},
$$
T_b^k(f)=\frac{d^k}{dz^k}T_z(f)|_{z=0}=\frac{k!}{2\pi
i}\int_{|z|=\ez} \frac{T_z(f)}{z^{k+1}}\,dz \qq \ez>0$$
by the Cauchy integral theorem. The same calculation as in the case
$k=1$ gives the required estimate, with $c_n^k\approx 2^{2nk}$, and
the same $\gamma_n= 4\beta_n$.

\end{proof}


\section{Examples}  \label{Examples}

In this section, we show that one can not have estimates better than
Theorem \ref{main}, Corollary \ref{COR1}, and Corollary \ref{COR2}.
We present  examples which return the same growth with respect to the
$A_p$ constant of the weight that appears in our results.
First, we discuss the
simpler case in dimension one. The following example shows that the
quadratic estimate  for the first commutator of the Hilbert
transform is sharp for $p=2$.

\subsection{Sharp example for the commutator of the Hilbert transform.}

Consider the Hilbert transform
\[
Hf(x) = p.v.\,\int_{\rr} \frac{f(y)}{x-y}\,dy,
\]
and consider the $BMO$ function $b(x)=\log |x|$. We know that there
is a constant $c$ such that
\begin{equation}
\|[b,H]\|_{L^2(w)} \leq\,c\, [w]_{A_2}^2
\end{equation}
and we show that the result is sharp. More precisely, for any increasing function
$\phi: [1,\infty)\to [0,\infty)$ such that $\displaystyle{\lim_{t\to\infty}\frac{t^2}{\phi (t)}=\infty}$ then
\begin{equation}\label{main1}
\sup_{w\in A_2} \frac{1}{\phi([w]_2)} \|[b,T]\|_{L^2(w)}
=\infty\,.
\end{equation}
In particular if $\phi(t)=t^{2-\epsilon}$ for any $\epsilon>0$, then (\ref{main1}) must hold.

For $0<\de<1$, we let $w(x) = |x|^{1-\de}$ and it is easy to see
that $[w]_{A_2} \sim 1/\de\,.$ We now consider the function
$$f(x) = x^{-1+\de} \, \chi_{(0,1)}(x)$$
and observe that $f$ is in $L^2(w)$ with
$\|f\|_{L^2(w)}=1/\sqrt{\de}\,.$ To estimate the $L^2(w)$-norm of
$[b,H]f\,,$ we claim
$$
|[b,H]f(x)| \geq \frac{1}{\de^2} \, f(x)
$$
and hence
$$\|[b,H]f\|_{L^2(w)} \geq \, \frac{1}{\de^2}\, \|f\|_{L^2(w)}
$$
from which the sharpness \eqref{main1} will follow.

We now prove the claim: if $0<x<1$,
$$
[b,H]f(x) = \int_0^1 \frac{\log(x)-\log (y)}{x-y}\,y^{-1+\de}\,dy=
\int_0^1 \frac{\log(\frac{x}{y})}{x-y}\,y^{-1+\de}\,dy
$$
$$
= x^{-1+\de} \, \int_0^{1/x}
\frac{\log(\frac{1}{t})}{1-t}\,t^{-1+\de}\,dt
$$
Now,
$$
\int_0^{1/x} \frac{\log(\frac{1}{t})}{1-t}\,t^{-1+\de}\,dt =
\int_0^1 \frac{\log(\frac{1}{t})}{1-t}\,t^{-1+\de}\,dt +
\int_1^{1/x} \frac{\log(\frac{1}{t})}{1-t}\,t^{-1+\de}\,dt
$$
and since $\frac{\log(\frac{1}{t})}{1-t}$ is positive for $(0,1)\cup
(1,\infty)$ we have for $0<x<1$
$$
|[b,H]f(x)| >  x^{-1+\de} \,\int_0^1
\frac{\log(\frac{1}{t})}{1-t}\,t^{-1+\de}\,dt.
$$
But since
$$
\int_0^1 \frac{\log(\frac{1}{t})}{1-t}\,t^{-1+\de}\,dt> \int_0^1
\log(\frac{1}{t})\,t^{-1+\de}\,dt= \int_0^{\infty}s\,e^{-s\de}\, ds=
\frac{1}{\de^2}
$$
and the claim
$$
|[b,H]f(x)| \geq \frac{1}{\de^2} \, f(x)
$$
follows. One can find this example and
 similar examples which show the commutators
with the Beurling-Ahlfors operator and the Riesz transforms obey the
quadratic growth  in \cite{Ch}.

\subsection{Sharp example for the $k$th-order commutator with the Riesz
transforms}

We now consider the higher order commutators with degree $k\geq 1$
in the case of the $j$-th directional Riesz transform on
$\mathbb{R}^n$:
\[
R^k_{j,b}f(x):=p.v.\int_{\mathbb{R}^n}\frac{(x_j-y_j)(b(x)-b(y))^kf(y)}{|x-y|^{n+1}}\,dy.
\]
%
To demonstrate the sharpness of the estimate for the higher order
commutator with the Riesz transforms, we show sharpness for $1<p\leq
2\,.$ Then we can extend the sharpness for all $1<p<\infty\,$ by
using a duality argument, because $R_j^{\ast}=-R_j\,,$  so the
higher order commutators of the Riesz transforms are  almost
self-adjoint operators. For $1<p\leq 2\,,$ we consider weights
$w(x)=|x|^{(n-\delta)(p-1)}\,,$
$f(x)={|x|}^{\delta-n}\chi_{E}(x)$ where
$E=\{y|y\in(0,1)^n\cap B(0,1)\}\,,$ $b(x)=\log|x|\,,$ and evaluate
$L^p(w)$-norm over $\Omega=\{x\in B(0,1)^c\,|\,
{x_i}<0\textrm{ for all }i=1,2,...,n\}\,.$ Note that, for all $y\in
E$ and $x\in\Omega$,
\[
|\,x_j-y_j|\geq |\,x_j|\textrm{ and } |x-y|\leq |y|+|x|\,.
\]
Then for $x\in\Omega$,
\begin{align*}
\big|R^{k}_{j,b}f(x)\big|&=\bigg|\int_E\frac{(x_j-y_j)(\log|x|-\log|y|)^k|y|^{\delta-n}}{|x-y|^{n+1}}\,dy\bigg|\\
&=\int_{E}\frac{|\,x_j-y_j|(\log(|x|/|y|))^k|y|^{\delta-n}}{|x-y|^{n+1}}\,dy
\geq |x_j|\int_E\frac{(\log(|x|/|y|))^k|y|^{\delta-n}}{(|y|+|x|)^{n+1}}\,dy\\
&= |\,x_j|\int_{E\cap
S^{n-1}}\int_0^1\frac{(\log(|x|/r))^kr^{\delta-n}r^{n-1}}{(r+|x|)^{n+1}}\,drd\sigma\\
&=c\,|\,x_j|\int_0^{1/|x|}\frac{(\log(1/t))^k(t|x|)^{\delta-1}|x|}{(|x|+t|x|)^{n+1}}\,dt
=\frac{c\,|\,x_j|}{|x|^{n+1-\delta}}\int_0^{1/|x|}\frac{(\log(1/t))^kt^{\delta-1}}{(1+t)^{n+1}}\,dt\\
&\geq
\frac{c\,|\,x_j|}{|x|^{n+1-\delta}}\bigg(\frac{|x|}{|x|+1}\bigg)^{n+1}\int_0^{1/|x|}(\log(1/t))^kt^{\delta-1}\,dt\,.
\end{align*}
Note that the constant $c=c(n)$  is the surface measure of $E\cap S^{n-1}$, depends on the dimension only.
On the other hand,
\begin{align*}
\int_0^{1/|x|}(\log(1/t))^kt^{\delta-1}\,dt&=\int_{\log|x|}^{\infty}
s^ke^{-\delta s}\,ds=-\frac{1}{\delta}e^{-\delta
s}s^k\bigg|^{\infty}_{\log|x|}+\frac{k}{\delta}\int_{\log|x|}^{\infty}e^{-\delta s}s^{k-1}\,ds\\
&=\frac{1}{\delta}e^{-\delta\log|x|}(\log|x|)^k+\frac{k}{\delta}\int_{\log|x|}^{\infty}e^{-\delta s}s^{k-1}\,ds\\
&=\frac{(\log|x|)^k}{\delta|x|^{\delta}}+\frac{k}{\delta}\int_{\log|x|}^{\infty}e^{-\delta
s}s^{k-1}\,ds\,.
\end{align*}
For $x\in\Omega$, $(\log|x|)^k/\delta|x|^{\delta}$ is positive,
therefore after neglecting some positive terms and applying the
integration by parts $k-1$ times, we get
\begin{align*}\int_0^{1/|x|}(\log(1/t))^kt^{\delta-1}\,dt
&\geq  \frac{k}{\delta}\int_{\log|x|}^{\infty}e^{-\delta s}s^{k-1}\,ds\\
&\geq
\cdots\geq\frac{k!}{\delta^{k-1}}\int_{\log|x|}^{\infty}e^{-\delta
s}s\,ds
\geq
\frac{k!}{\delta^{k+1}|x|^{\delta}}\,.
\end{align*}
Combining the previous computations, we have that for
$x\in\Omega$, and recalling that
$f(x)={|x|}^{\delta -n}\chi_E(x)$,
\[
\big|R^{k}_{j,b}f(x)\big|\geq \frac{k!\,
c \,  |\,x_j|}{\delta^{k+1}(|x|+1)^{n+1}}\,. \] Thus, we estimate
\begin{align*}
\big\|R^k_{j,b}f\big\|^p_{L^p(w)}&\geq\int_{\Omega}\bigg(\frac{k!\, c\,|\,x_j|}{\delta^{k+1}(|x|+1)^{n+1}}\bigg)^p|x|^{(n-\delta)(p-1)}\,dx\\
&=\bigg(\frac{k! \, c}{\delta^{k+1}}\bigg)^p\int_{\Omega}\frac{|\,x_j|^p|x|^{(n-\delta)(p-1)}}{(|x|+1)^{p(n+1)}}\,dx\\
&\geq\bigg(\frac{k! \, c}{\delta^{k+1}}\bigg)^p\int_{\Omega\cap
S^{n-1}}\int_1^{\infty}\frac{\gamma_j^pr^pr^{(n-\delta)(p-1)}r^{n-1}}{(r+1)^{p(n+1)}}\,drd\sigma(\gamma)\\
&=\bigg(\frac{k! \, c_{n,p}}{\delta^{k+1}}\bigg)^p\int_1^{\infty}r^{\delta(1-p)-1}\,dr=(k!)^p\frac{c^p_{n,p}}{p-1}\delta^{-p(k+1)-1}\,.
\end{align*}
Since $\|f\|^p_{L^p(w)}=1/\delta\,,$  and $[w]_{A_p}\sim
1/\delta^{p-1}$, we conclude that
\[
\big\|R^{k}_{j,b}f\big\|_{L^p(w)}\geq k! \,\frac{c_{n,p}}{(p-1)^{1/p}}\,\delta^{-(k+1)-1/p}
\sim k!\,[w]^{\frac{k+1}{p-1}}_{A_p}\|f\|_{L^p(w)}\,.
\]
This shows that Theorem \ref{main}, Corollary \ref{COR1}, and
Corollary \ref{COR2} are sharp in the multidimensional case.
The constant $c_{n,p}\to c_{n,1} >0$ as $p\to 1$, therefore the estimate
blows up as $p\to 1$, as it should since the operators are not bounded in $L^1(w)$.

\section{Appendix: the sharp reverse H\"older's inequality for $A_2$
weights}\label{append}

In this section we give a proof of Lemma \ref{RHA2}, namely if $w\in
A_2$, then
\begin{equation}\label{a3}
\left ( \frac{1}{|Q|}\int_Qw^{r_w}dx\right )^{\frac{1}{r_w}} \le
\frac{2}{|Q|}\int_Qw,
\end{equation}
where $r_{w}=1+\frac{1}{2^{5+n}[w]_{A_2}}$.

\begin{proof}

Let $w_Q=\frac{1}{|Q|}\int_Qw$ and $\delta>0$
\[
\frac{1}{ |Q| } \int_Q  w(x)^{1+\delta}\,  dx =  \frac{1}{ |Q| }
\int_Q  w(x)^{\delta}\,  w(x)dx = \frac{\delta}{|Q|}
\int_{0}^{\infty} \lambda^{\delta} w( \{ x \in Q : w(x) > \lambda \}) \,
\frac{d\lambda}{\lambda}
\]
\[
= \frac{\delta}{|Q|} \int_{0}^{w_Q} \lambda^{\delta} w( \{ x \in Q : w(x)
> \lambda \}) \,\frac{d\lambda}{\lambda}
+ \frac{\delta}{|Q|} \int_{w_Q}^{\infty} \lambda^{\delta} w( \{ x \in Q :
w(x)>\lambda \}) \, \frac{d\lambda}{\lambda}
 = I+II.
\]
Observe that $I\le (w_Q)^{\delta+1}$, where $w_Q= \frac{w(Q)}{|Q|}$.

To estimate $II$ we make two claims. The first is the following
observation: if we let
$$E_Q=\{x\in Q: w(x)\le  \frac{1}{2 [w]_{A_2}}w_Q\},$$
then
\begin{equation}\label{claimI}
|E_Q| \le \frac12 |Q|.
\end{equation}
Indeed, by (Cauchy-Schwartz) we have for any $f\ge 0$
$$ \left( \frac{1}{ |Q| } \int_{Q} f(y)\, dy \right)^{2} w(Q)\le
[w]_{A_2} \int_{Q} f(y)^{2}\, w(y)dy,
$$
and hence if $E\subset Q$, setting $f=\chi_E$,
$$ \left( \frac{|E|}{ |Q| } \right)^{2} \le
[w]_{A_2}\frac{w(E)}{w(Q)}
$$
and in particular, by definition of $E_Q$,
$$ \left( \frac{|E_Q|}{ |Q| } \right)^{2} \le
[w]_{A_2}\frac{w(E_Q)}{w(Q)} \le [w]_{A_2}\frac{w_Q}{w(Q)} |E_Q|
\frac{1}{2 [w]_{A_2}}=\frac{1}{2}\frac{|E_Q|}{|Q|},
$$
from which the claim follows. In particular, this implies that

\begin{equation}\label{claimIforcomplement}
 |Q|\leq 2|Q\backslash E_Q|=
    2 |\{ x\in Q: w(x)> \frac{1}{2 [w]_{A_2}}w_Q\} |.
\end{equation}

The second claim is the following
\begin{equation}\label{claimII}
w(\{x\in Q: w(x)>\la \}) \le 2^{n+1} \la \, |\{x\in Q: w(x)>
\frac{\la}{2 [w]_{A_p}}w_Q\}| \qquad \la > w_Q.
\end{equation}
Indeed, since $\la > w_Q$ to prove this claim we consider the
standard (local) Calde´r\'on-Zygmund decomposition of $w$ at level
$\la$. Then there is a family of disjoint cubes $ \{Q_{i}\} $
contained in $Q$ satisfying
\[
\la <  w_{Q_i}  \le 2^n\,\la
\]
for each $i$. Now, observe that except for a  null set we have
$$\{x\in
Q: w(x)>\la \} \subset \{x\in Q: M_Q^dw(x)>\la \}= \cup_i Q_i,$$
where $M^{d}_Q$ is the dyadic maximal operator restricted to a cube
$Q$.
This together with \eqref{claimIforcomplement} yields
\[
w(\{x\in Q: w(x)>\la \}) \le \sum_{ i } w(Q_{i})
\]
\[
\le 2^n\la\, \sum_i  |Q_{i}| \le 2^{n+1}\la\, \sum_i |\{x\in Q_i:
w(x)>  \frac{1}{2 [w]_{A_2}}w_{Q_i}\}|
\]
\[
\le 2^{n+1} \la\, |\{x\in Q: w(x)>  \frac{1}{2 [w]_{A_2}}\la\}|
\]
since $ w_{Q_i}> \la$. This proves the second claim \eqref{claimII}.

Now, combining

$$
II= \frac{\delta}{|Q|} \int_{w_Q}^{\infty} \lambda^{\delta} w( \{ x \in Q
: w(x)>\la \}) \, \frac{d\la}{\la}
$$
\[
\le  \frac{2^{n+1}\,\delta}{|Q|} \int_{w_Q}^{\infty} \la^{\delta +1}
|\{x\in Q: w(x)>  \frac{1}{2 [w]_{A_2}}\la\}| \, \frac{d\la}{\la}
\]
\[
\leq  \left(2 [w]_{A_2}\right)^{1+\delta}  2^{n+1}\,\delta
\frac{1}{|Q|} \int_{ \frac{w_Q}{2 [w]_{A_2}} }^{\infty} \la^{\delta
+1} |\{ x \in Q : w(x) > \la \}| \, \frac{d\la}{\la}
\]
\[
\le \left(2 [w]_{A_2}\right)^{1+\delta} 2^{n+1}
\frac{\delta}{1+\delta}\frac{1}{|Q|} \int_{Q} w^{1+\delta}\,dx.
\]
 Setting here $\delta= \frac{1}{2^{5+n}[w]_{A_2}}$, we obtain using
that $t^{1/t}\le 2$, $t\ge 1$
$$II\leq \frac12 \frac{1}{|Q|}\int_Qw^{\delta+1}dx$$
and finally
$$\frac{1}{|Q|}\int_Qw^{\delta+1}dx\le
2(w_Q)^{\delta+1},$$
which proves \eqref{a3}.

\end{proof}

\end{document}